\documentclass[12pt]{amsart}
\makeindex
\usepackage[latin1]{inputenc} 
\usepackage{lmodern}
\usepackage{epsfig,graphicx,color,amsmath,a4wide,amssymb,amsfonts,amsbsy,xy,latexsym,epsf,pstricks,verbatim,times,hyperref}
\usepackage{hyperref}
\usepackage{enumerate}
\usepackage{pgf,color}
\usepackage{pgfarrows}

\xyoption{all}

\xyoption{all}
\usepackage[english]{babel}

\newcounter{noalgo}[section]
\newdimen\indentalgo
\newdimen\indentalgodec\indentalgo=0.0mm\indentalgodec=10mm

\newcommand{\If}{\advance\indentalgo by \indentalgodec {\bf if }}

\newcommand{\For}{\global\advance\indentalgo by \indentalgodec {\bf for }}

\newcommand{\Endindent}{\global\advance\indentalgo by -\indentalgodec}

\newdimen\decalage \decalage=0.5cm
\newcounter{algo} 

\def\<<{\leavevmode
  \raise0.28ex\hbox{$\scriptscriptstyle\langle\!\langle$}\nobreak
  \hskip -.6pt plus.3pt minus.2pt\,}
\def\>>{\,\nobreak\hskip -.6pt plus.3pt minus.2pt
  \raise0.28ex\hbox{$\scriptscriptstyle\rangle\!\rangle$}}

\def\<<{\leavevmode
  \raise0.28ex\hbox{$\scriptscriptstyle\langle\!\langle$}\nobreak
  \hskip -.6pt plus.3pt minus.2pt\,}
\def\>>{\,\nobreak\hskip -.6pt plus.3pt minus.2pt
  \raise0.28ex\hbox{$\scriptscriptstyle\rangle\!\rangle$}}

\newtheorem{remark}{Remark}
\newtheorem{theorem}{Theorem}

\newtheorem{lemma}{Lemma}

\def\Gal{{\mathop{\rm Gal}\nolimits}}

\providecommand{\myproofname}{Proof}

\begin{document}

\title{Minimal Hopf-Galois Structures on Separable Field Extensions}

\author{Tony Ezome}
\address{Universit{\'e} des Sciences et Techniques de Masuku,
Facult{\'e} des Sciences, D{\'e}partement de math{\'e}matiques et informatique,
BP 943 Franceville, Gabon.}
\email{tony.ezome@gmail.com}

\author{Cornelius Greither}
\address{Cornelius Greither, Institut f{\"u}r Theoretische Informatik, 
Mathematik und Operations Research, Fakult{\"a}t f{\"u}r 
Informatik, Universit{\"a}t der Bundeswehr M{\"u}nchen,
Werner-Heisenberg-Weg 39 
85579 Neubiberg }
\email{cornelius.greither@unibw.de}


\date{\today}

\maketitle

\bibliographystyle{plain}

\begin{abstract}
In Hopf-Galois theory, 
every $H$-Hopf-Galois structure on a field extension $K/k$
gives rise to an injective
map $\mathcal{F}$ from the set of $k$-sub-Hopf algebras of $H$
into the intermediate fields of $K/k$.
Recent papers on the failure of the surjectivity of $\mathcal{F}$
reveal that there exist many Hopf-Galois structures for which there
are many more subfields than sub-Hopf algebras.
This paper surveys and illustrates group-theoretical methods
to determine $H$-Hopf-Galois structures on finite separable extensions
in the extreme situation when $H$ has only
two sub-Hopf algebras.
\end{abstract}


\section{Introduction}

Let $k$ be field.
A Hopf algebra $H$ over $k$ is defined to be a $k$-bialgebra endowed with a $k$-linear
map $ S: H \longrightarrow H$ called the antipode so that denoting by
$\nabla$ the multiplication, $\Delta$ the comultiplication, 
$\eta$ the unit and $\epsilon$ the counit, we have
$$\nabla \circ (\mathrm{id_H} \otimes
S) \circ \Delta= \eta \circ \epsilon
= \nabla \circ (S  \otimes \mathrm{id_H}) \circ \Delta.$$
Let $\sigma : H \otimes H \longrightarrow H\otimes H$ be the $k$-linear map 
defined by $\sigma(x \otimes y) = y\otimes x$ for all $x,y \in H$.
Then, $H$ is said to be cocommutative if $\sigma \circ \Delta=\Delta$.
Group algebras over $k$ are basic examples of cocommutative $k$-Hopf algebras.
Indeed if $G$ is a group, then the group algebra $k[G]$ is
a cocommutative $k$-Hopf algebra with comultiplication
given by $\Delta(g) = g \otimes g$, counit given by $\epsilon(g) = 1$
and antipode given by $S(g) = g^{-1}$, for all $g \in  G$.
Given a Galois extension of fields $K/k$, the Fundamental Theorem of Galois Theory (FTGT)
states that there is a one-to-one correspondence between 
the lattice of intermediate fields $k\subseteq F \subseteq K$
and the lattice of subgroups of $G=\Gal(K/k)$.
This is the Galois correspondence. It allows us to determine
intermediate subfields of $K$ from subgroups of $G$.
So if $G$ is a group with prime order,
then the only subfields are $K$ and $k$.
Hopf-Galois theory is a generalization of Galois theory.
Indeed if $K/k$ is Galois
with Galois group $G$, then $G$ operates linearly on $K$
as automorphism group, and this action
extends to a $k$-algebra homomorphism 
$\mu : k[G] \longrightarrow \mathrm{End}_k(K)$ so that:
$$K/k \text{ is Galois } \Longleftrightarrow
(1,\mu) : K \otimes_k k[G] \longrightarrow \mathrm{End}_k(K) \text{ is
an isomorphism,}$$
where  $(1,\mu)$ is given by
$$ (1,\mu)(s\otimes h)(t)=s.(\mu(h)(t)), \ 
 \text{ for all } s,t \in K, h \in k[G].$$
From this we say that a finite extension of fields $K/k$
is \textit{Hopf-Galois} (we also say that $K/k$ has a 
\textit{Hopf-Galois structure}) if there exists
a finite cocommutative $k$-Hopf algebra
$H$ and a Hopf action $\mu  : H \longrightarrow \mathrm{End}_k(K)$
such that
$$(1, \mu ) : K \otimes_k H \longrightarrow \mathrm{End}_k(K) \text{ is
an isomorphism.}$$

Chase and Sweedler obtained a weak Galois correspondence 
for Hopf-Galois extensions.

\begin{theorem}[\cite{ChaseSweedler}]\label{thm:1}
 Let $K/k$ be a finite Hopf-Galois extension with algebra $H$
and Hopf action $\mu  : H \longrightarrow \mathrm{End}_k(K)$.
For a $k$-sub-Hopf algebra $H'$ of $H$ we define
$$K^{H'} = \{x \in K \ \vert \ \mu(h)(x) = \epsilon(h) \cdot x \
\text{ for all } h \in H' \},$$
where $\epsilon$ is the counit of $H$.
Then, $K^{H'}$ is a subfield of $K$, containing $k$, and the map
$$ \begin{array}{llll}
\mathcal{F} : & \{ H'\subset H \text{ sub-Hopf algebra} \} & \longrightarrow & 
\{ \text{Fields } E \ \vert \ k\subseteq E \subseteq K \}\\
 & \qquad \qquad \qquad \qquad \qquad H'  & \longmapsto &  \ K^{H'} 
\end{array}$$
is injective and inclusion reversing.
\end{theorem}

Recent papers on the failure of the surjectivity of $\mathcal{F}$ 
reveal that pretty often there are many more subfields than 
sub-Hopf algebras, see for instance \cite{ChildsGreither}, \cite{Childs2},
or \cite{CrespoRioVela2}.
We say that {\it the Galois correspondence holds in its strong form}
for a Hopf-Galois structure $H$ on a field extension $K/k$,
if the map $\mathcal{F}$ associated to $H$ in Theorem $\ref{thm:1}$ is a bijection.
It is known that the $k$-sub-Hopf algebras of a finite group algebra
$k[G]$ are the group algebras
$k[G']$ where $G'$ is a subgroup of $G$, see for instance
 [\cite{CrespoRioVela2}, Proposition 2.1]. Therefore,
FTGT implies that any finite Galois extension $K/k$ with Galois group $G$
has a natural Hopf-Galois structure (defined by the group algebra $k[G]$) whose
Galois correspondence holds in its strong form.
In case $G$ has prime order, $k[G]$ has only two $k$-sub-Hopf algebras.
Motivated by this, we define a {\it minimal Hopf-Galois structure} on a field extension $K/k$
to be a structure given by a $k$-Hopf algebra $H$ having exactly two $k$-sub-Hopf algebras;
we exclude the trivial case $\dim_k(H) = 1$.

This paper surveys and illustrates group-theoretical methods
to determine minimal Hopf-Galois structures on separable field extensions. 
In section $2$ we state a fundamental criterion 
characterizing these minimal
structures. Then we deduce, later in section $4$, minimal Hopf-Galois
structures on the so-called {\it almost classically Galois extensions} introduced by
Greither and Pareigis in \cite{GreitherPareigis}. Sections $3$ and $5$ are devoted to illustrations.
We start with basic examples constructed from simple groups, and
counterexamples constructed from groups having a nontrivial proper characteristic subgroup.
In particular, we present a family of
radical extensions in characteristic zero having no Hopf-Galois structure.
By using characteristically simple groups, we prove that for any positive integer
$n \le 9$, except for $ n = 6$, there exists a number field $K$ of degree $n$
whose Galois closure $\tilde{K}$ satisfies $n < [\tilde{K} : \mathbb{Q}] < 672$
and such that $K/\mathbb{Q}$ has only one minimal Hopf-Galois structure.
All these examples are separable field extensions having either no minimal Hopf-Galois structure,
or exactly one minimal structure,
or at least two minimal structures.
We deduce interesting questions for future work.

\section{Fundamental criterion}\label{section:Basic}

As previously mentioned, the Galois correspondence associated to
an Hopf-Galois structure  is not surjective in general.
Another difference between Galois theory and Hopf-Galois
theory is that
one may have several Hopf-Galois structures on the same Galois extension
while a Galois extension has only one Galois group.
Hopf-Galois theory was first introduced by Chase and Sweedler
\cite{ChaseSweedler} in 1969 to study purely inseparable extensions.
Then Greither and Pareigis \cite{GreitherPareigis} developed  in 1987 Hopf-Galois
theory for separable extensions. Since the publication of \cite{GreitherPareigis},
many works concerning Hopf-Galois theory have been published.
These works deal with interesting problems such as designing methods
to determine the number of distinct Hopf-Galois structures on
a given Galois extension, finding ways of quantifying Hopf-Galois structures
for which Galois correspondence holds in its strong form, or finding ways
of quantifying the failure of the surjectivity of the Galois correspondence.\\

In this section we are interested in identifying
minimal Hopf-Galois structures among the structures 
that can be achieved on a given
separable field extension. The starting point is the characterization
of Hopf-Galois structures proposed by Greither and Pareigis.

\begin{theorem}[\cite{GreitherPareigis}, Theorem 2.1]\label{thm:2}
Let $K/k$ be a degree $n$ separable extension
and $\tilde{K}$ its normal closure. Set $G = \mathrm{Gal}(\tilde{K}/k)$
and $G' = \mathrm{Gal}(\tilde{K}/K)$. Then 
$K$ has a $k$-Hopf-Galois structure if, and only if, there exists a 
regular subgroup $N$ of $\mathrm{Perm}(G/G')$ normalized by 
$G$, where $G$ is identified as a subgroup of $\mathrm{Perm}(G/G')$
via the faithful action
$$ \xymatrix{
 \lambda : G  \ar@{->}[r] & 
 \mathrm{Perm}(G/G')\\
    g  \ar@{|-{>}}[r] & (\lambda_g : xG'\mapsto gxG').
}$$
Furthermore, the Hopf-Galois structure corresponding to
a regular subgroup $N$ of $\mathrm{Perm}(G/G')$ normalized
by $G$ is defined by 
$$ \tilde{K}[N]^{G}
=\{ x\in \tilde{K}[N] \ \vert \ \sigma(x)=x, \forall \sigma \in G \}$$
where for $x=\sum_{\tau \in N} a_\tau \tau \in \tilde{K}[N]$
and $\sigma \in G$, we have $\sigma(x)= \sum_{\tau \in N} 
\sigma(a_\tau) \lambda(\sigma) \tau \lambda(\sigma)^{-1}$.
\end{theorem}

With the notation of Theorem \ref{thm:2},
a Hopf-Galois structure on a separable field extension $K/k$ defined
by the algebra $\tilde{K}[N]^{G}$ is said to be of type $N$.
Actually, $N$ and $G/G'$ have necessarily the same order,
  but there is no natural one-one correspondence between them;
	in fact $G/G'$ is not even a group in general.
We recall that a subgroup $N$ of
$\mathrm{Perm}(G/G')$ is said to be regular if the action of $N$ on $G/G'$
is transitive and the stabilizer of any point is trivial.
By [\cite{GreitherPareigis}, Theorem 4.1], we know that if such an $N$
is also normalized by $G$ and contained in $G$ then
it is a normal complement of $G'$ in $G$. 
Theorem \ref{thm:2} says that regular subgroups
of $\mathrm{Perm}(G/G')$ normalized by $G$
are in one-to-one correspondence with the
Hopf-Galois structures on $K/k$.
The following theorem specifies the minimal Hopf-Galois structures inside this correspondence.

\begin{theorem}[Fundamental criterion]\label{thm:3}
Let $K/k$ be a finite separable extension. 
Let $\tilde{K}$ be the normal closure of $K/k$.
Set $G = \mathrm{Gal}(\tilde{K}/k)$ and $G' = \mathrm{Gal}(\tilde{K}/K)$.
Then the minimal Hopf-Galois structures on
$K/k$ are defined by the algebras $\tilde{K}[N]^{G}$ for which $N$
is a regular subgroup of $\mathrm{Perm}(G/G')$
normalized by $G$ such that $N$ has no proper nontrivial
subgroup normalized by $G$. 
In particular:
\begin{enumerate}[1.]
\item The number of minimal Hopf-Galois structures on $K/k$ is greater 
than or equal to the number of normal complements $N$ of $G'$
in $G$ such that $N$ admits no proper nontrivial subgroup $U$ which
is a normal subgroup of $G$.
\item Assume that $K/k$ has a Hopf-Galois structure defined by $\tilde{K}[N]^{G}$.
   \begin{enumerate}
   \item If $N$ has a nontrivial proper characteristic subgroup
(this is the case when $K/k$ is a Hopf-Galois extension of degree $mp$
where $p$ is a prime number and $p > m>1$), then
this structure is not minimal.
   \item If $N$ has prime order, then
the structure is minimal.
   \end{enumerate}
\item If $K/k$ is a Galois extension whose Galois group is a simple
group, then $K/k$ has only one minimal Hopf-Galois structure.
\end{enumerate}
\end{theorem}

\begin{proof}
Since the Hopf algebras providing a Hopf
Galois structure on $K/k$ are of the form $\tilde{K}[N]^G$,
the assertion results from [\cite{CrespoRioVela2}, Proposition 2.2]. 
\begin{enumerate}[1.]
\item This special case is an immediate consequence of [\cite{GreitherPareigis}, Proposition 4.1].

\item 
   \begin{enumerate}
   \item Let $N$ be a group of order $mp$ where $p$ is a prime number and $p > m>1$.
Then  the unique $p$-Sylow of $N$ is a nontrivial proper characteristic subgroup.
So a Hopf-Galois extension of degree $mp$
with $p$ a prime number such that $p > m>1$ is a 
special case of the situation that we are interested in.
Assume now that $N$ is a regular subgroup of $\mathrm{Perm}(G/G')$
normalized by $G$. So $\lambda(x) N \lambda(x)^{-1} \subset N$ for all $x \in G$,
where $\lambda$ is the faithful action
described in Theorem \ref{thm:2}. Assume also that $N$ possesses at least one nontrivial proper characteristic subgroup $U$.
Since $G$ normalizes $N$, the maps $n \mapsto \lambda(x) n \lambda(x)^{-1}$
are automorphisms of $N$. We deduce that $G$ also normalizes $U$. 
Hence $k$, $\tilde{K}[U]^G$ and $\tilde{K}[N]^G$ are distinct $k$-sub-Hopf algebras
of  $\tilde{K}[N]^G$ by [\cite{CrespoRioVela2}, Proposition 2.2].

   \item In that case, $N$ has no nontrivial proper subgroup.
Therefore the only $k$-sub-Hopf algebras of $\tilde{K}[N]^{G}$ are
$k$ and $\tilde{K}[N]^{G}$ itself.
   \end{enumerate}

\item Assume first that the Galois group $G$ of $K/k$ is an abelian simple group. This
means that $G$ is a cyclic group with prime order.
From the above item, we deduce that the classical Hopf-Galois structure
defined by the group algebra $k[G]$ is a minimal one.
This is the only Hopf-Galois structure on $K/k$
by [\cite{Byott3}, Theorem 1].
On the other hand, assume that $K/k$ is a Galois extension 
whose Galois group $G$ is a nonabelian simple group.
By [\cite{Byott1}, Theorem 1.1],
there are exactly two Hopf-Galois structures on $K/k$.
By [\cite{GreitherPareigis}, Theorem 5.3], we know that one of these
structures comes from a Hopf algebra $H$
giving rise to a bijective Galois correspondence between
its $k$-sub-Hopf algebras and intermediate subfields
$k\subseteq F \subseteq K$ which are normal over $k$.
However, $G$ is a simple group,
therefore the only subfields which are normal over $k$ are $k$ itself and $K$.
Thus, this Hopf-Galois structure is minimal. The other Hopf-Galois
structure is the classical one, and it is obviously not minimal,
since $G$ (nonabelian simple) does have nontrivial subgroups.
That is, for nonabelian simple $G$ as well, we have only one minimal
Hopf-Galois structure.
\end{enumerate}
\end{proof}

\begin{remark}
Concerning item $3$ of Theorem \ref{thm:3},
we would like to point out that one has precise information
about the only two Hopf-Galois structures [\cite{Byott1}, Theorem 1.1] defined
on a Galois extension $K/k$ with nonabelian simple Galois group $G$.
Indeed, one of them, the classical one,
  is given by the group algebra $H=k[G]$. The other one
	(the first one to be considered in the last paragraph) arises
	by taking $N=\lambda(G)$. Hence the action of $G$ on $N$ amounts
	to the conjugation action of $G$ on itself. The $k$-Hopf algebra
	$H'$ which results may be
	constructed 
	for any $G$, and as soon as $G$ is not abelian, $H'$ is not isomorphic
	to $H$ as a $k$-Hopf algebra.
\end{remark}

\section{Examples (part 1)}

This section illustrates some of the minimal Hopf-Galois structures described in Theorem 
\ref{thm:3}. 

\subsection{Example 1}\label{exple1}
Let $K/k$ be a separable extension of degree $n\le 4$ whose normal closure $\tilde{K}/k$
has Galois group $G$. Set $G' = \mathrm{Gal}(\tilde{K}/K)$.
Assume that $G'$ has a normal complement in $G$ and $\mathrm{Perm}(G/G')$
is isomorphic to $G$. Then $K/k$ has only one minimal Hopf-Galois structure. Indeed:
\begin{enumerate}[1.]
\item Assume $n=2$. Since any separable extension of degree $2$ is Galois,
our assertion comes from the third item of Theorem \ref{thm:3}.
\item Assume $n=3$. Then $G$
is isomorphic to the symmetric group $S_3$.
The algebra $H=\tilde{K}[C_3]^{S_3}$ defines 
the only minimal Hopf-Galois structure on $K/k$.
\item Assume $n=4$. Then $G$ is isomorphic to $S_4$.
Since:
   \begin{enumerate}
   \item The Klein four-group $C_2 \times C_2$ is 
the unique normal subgroup of $S_4$ of order $4$,
   \item $C_2$ is the unique proper nontrivial subgroup of the Klein four-group,
   \item $C_2$ is not a normal subgroup of $S_4$,
   \end{enumerate}
we conclude that $\tilde{K}[C_2 \times C_2]^{S_4}$ defines 
the only minimal Hopf-Galois structure on $K/k$.
\end{enumerate}

\begin{remark}
If $K$ is a number field of degree $5$ such that
its normal closure $\tilde{K}/\mathbb{Q}$ has Galois group
$S_5$, then $K/\mathbb{Q}$ has no Hopf 
Galois structure because $S_5$ admits no normal
subgroup of order $5$. There is another argument to see this
in a more general way, see for instance
[\cite{GreitherPareigis}, Proof of Counterexample 2.4].
\end{remark}

\subsection{Example 2}\label{exple2}
Hopf-Galois extensions without minimal structure.
\begin{enumerate}[1.]
\item Given an odd prime number $p$, it is easily seen that:
\begin{enumerate}
\item No dihedral extension of degree $2p$
can have a minimal Hopf-Galois structure.
\item No Galois extension whose Galois group is equal to the holomorph of
the cyclic group $C_p$ can have a minimal Hopf-Galois structure.
\end{enumerate}

\item \begin{enumerate}
\item Let $p$ be an odd prime and 
$n$ a positive integer. Let $k$ be a field of characteristic zero. 
Assume $K=k(w)$ with $w^{p^n}=a \in k$ where $a$ is such that 
$[K:k ]=p^n$ and let $r$ denote the largest integer between $0$ and $n$
such that $K\cap k(\zeta_{p^r})=k(\zeta_{p^r})$, where $\zeta_{p^r}$ denotes a primitive
$p^r$-th root of unity. It is shown in \cite{Kohl1} that if $r <n$
then there are $p^r$ Hopf-Galois structures on 
$K/k$ of type $N$, a cyclic group of order $p^n$. 
So if $n\ge 2$, then none of these $p^r$ Hopf-Galois structures is a minimal one,
since $N$ does have characteristic subgroups.

\item Assume that $K/k$ is a cyclic extension of degree $2^n$ for $n \ge 3$.
It is shown in \cite{Byott4} that $K/k$ admits $3\cdot 2^{n-2}$ Hopf
Galois structures. Among them $2^{n-2}$ of cyclic type, $2^{n-2}$ 
of dihedral type and $2^{n-2}$ of generalized quaternion type.
In fact, any of these structures is associated to a subgroup $N$ of $\mathrm{Perm}(G/G')$
which has at least one nontrivial proper characteristic subgroup. Indeed:
\begin{itemize}
\item If $N$ is a cyclic group, then any subgroup of $N$ is a characteristic subgroup.
In the case when $N$ has order $2^n$ for $n \ge 3$, there are at least $2$ 
nontrivial proper characteristic subgroups.
\item If $N=D_{2^n}$ is a dihedral group, then its unique normal subgroup of order $2^{n-1}$
is a nontrivial proper characteristic subgroup. 
\item If $N=Q_n$ is a generalized quaternion group, then its unique normal subgroup of order $2$
is a nontrivial proper characteristic subgroup.
\end{itemize}
\end{enumerate}
\end{enumerate}

\section{Minimal Hopf-Galois structures on almost classically Galois extensions}\label{section:almost classically}

As before, we consider a separable field
extension $K/k$ of degree $n$, and we denote by $\tilde{K}/k$ its normal closure. Set
$G$ the Galois group of $\tilde{K}/k$ and $G'$ the Galois group of $\tilde{K}/K$.
We previously determined in Theorem \ref{thm:3} a lower bound of the 
number of minimal Hopf-Galois structures on $K/k$ from 
normal complements of $G'$ in $G$.
By [\cite{GreitherPareigis}, Definition 4.2], the existence
of a normal complement $N$ of $G'$ in $G$ means
that $K/k$ is an almost classically Galois extension.
This is equivalent to saying that $G$ is 
equal to the semidirect product $G=N\rtimes_{\varphi} G'$.
Note that any almost classically Galois extension has a Hopf-Galois structure.
That is why
these extensions are sometimes called almost classically Hopf-Galois extensions.
If the normal complement $N$ of $G'$ in $G$ is a cyclic group,
one says that $K/k$ is an \textit{almost cyclic extension}, see \cite{Byott4}.
Note that any Galois extension $K/k$ is obviously
almost classically Galois with $N=\mathrm{Gal}(K/k)$ and $G'=\{1\}$.
In this section we are interested in minimal Hopf-Galois structures on 
almost classically Galois extensions $K/k$ in the case when the Galois group
$\mathrm{Gal}(\tilde{K}/k)$ is equal to the holomorph of one of its normal subgroup.

 The inverse Galois problem in
Galois theory is concerned with the question of 
determining whether, given a finite group $G$ and a field $k$,
there exists a Galois extension $M/k$ such
that the Galois group $\mathrm{Gal}(M/k)$ is isomorphic to $G$.
If that is the case, one says that $G$ is realizable over $k$.
The classical conjecture of the inverse Galois problem says that every finite group 
is realizable over the rational numbers. 
This conjecture has been formulated in the early 19th century.
It is still not proven, but partial results have been obtained. For instance, 
Igor Shafarevich showed that every finite solvable group is realizable over $\mathbb{Q}$.
On the other hand, it is known that every finite group
is realizable over $\overline{\mathbb{Q}}(t)$ and more generally over function fields in one 
variable over any algebraically closed field of characteristic zero.
So any semidirect product $N\rtimes_{\varphi} G'$ is realizable at least over
$k=\overline{\mathbb{Q}}(t)$.

\begin{lemma}\label{lemma1}
With the above notation, assume that $K/k$ is an
almost classically Galois extension
such that $G$ is the holomorph of a characteristically simple group
$N$. Then $\tilde{K}[N]^{G}$ defines a minimal Hopf-Galois structure on $K/k$.
\end{lemma}

\begin{proof}
By [\cite{CrespoRioVela2}, Proposition 2.2], the Hopf-Galois structure on  $K/k$
defined by $\tilde{K}[N]^{G}$ is minimal if $N$ has no
proper nontrivial subgroup $U$ which is a normal subgroup of 
$G=N\rtimes_{\varphi} G'$.
This means that there is no $U$ such that
$$[i_a \circ \varphi(b) ](U)=U, \ \text{ for all }  a \in N, \ b \in G'$$
where $i_a$ stands for the inner automorphism of $N$ associated to $a$.
This is equivalent to saying that $\tilde{K}[N]^{G}$ defines a minimal structure
if there is no $U$ invariant under $G'$.
In particular, if $N$ is a characteristically simple group and $G$ is the holomorph of $N$,
then $\tilde{K}[N]^{G}$ defines a minimal Hopf-Galois structure on $K/k$.
\end{proof}

Simple groups obviously form a proper subfamily of characteristically simple groups.
On the other hand, Galois extensions are almost classically Galois.
We thus obtain more minimal Hopf-Galois structures from
the study made in this section than the one made in Section
\ref{section:Basic}. Note that any characteristically simple group is the direct sum
of finitely many copies of some simple group (see for instance [\cite{Byott1}, Lemma 3.2],
[\cite{Robinson}, 3.3.15] or [\cite{J.S.Rose}, Theorem 8.10]). 
The Klein four-group is the smallest abelian characteristically simple group which is not simple.
Besides, the direct product $A_5 \times A_5$ is the smallest non-abelian
characteristically simple group which is not simple. We already described
 in Example \ref{exple1} minimal
Hopf-Galois structures by using the Klein four-group and subgroups
of symmetric groups.
Lemma \ref{lemma1}  allows us to construct even more examples.

\section{Examples (part 2)}

This section illustrates minimal Hopf-Galois structures described in Lemma
\ref{lemma1} and Theorem \ref{thm:3}. 

       \subsection{Example 3}\label{exple:3}
We are interested in almost classically
Hopf-Galois extensions $K/k$ of degree $n$ such that
$n \le 9$, or $n=2^r$ and $r\ge 2$.
\begin{enumerate}[1.]
\item Burnside's theorem in Group Theory states that if
$G$ is a finite group of order $p^\alpha q^\beta$
where $p$ and $q$ are prime numbers, and $\alpha$ and $\beta$
are non-negative integers, then $G$ is solvable.
On the other hand, Shafarevich showed that that every 
finite solvable group is realizable over $\mathbb{Q}$.
Even if for $n \ge 5$
the symmetric group $S_n$ and the alternating group $A_n$ are not solvable,
Hilbert proved that for any positive integer $n$,
the symmetric group $S_n$ and the alternating group $A_n$ are realizable over
$\mathbb{Q}$. In addition, Sonn showed in \cite{Sonn} that every finite group of order
less than $672$ is realizable over $\mathbb{Q}$.
Hence for any positive integer $ n\le 9$, except for $n=6$, there exists
a number field $K$ of degree $n$ having only one minimal Hopf-Galois structure
and whose normal closure satisfies $n<[\tilde{K} : \mathbb{Q}]< 672$. Indeed:
   \begin{enumerate}
   \item Assume that $n$ is a prime number $\le 9$. Then the dihedral group $D_n$ is realizable
over $\mathbb{Q}$. Let $K/\mathbb{Q}$ be a dihedral extension
with Galois group $D_n$. Then the fixed field of $C_2$ 
is a number field of degree $n$ having only one minimal Hopf-Galois structure.

   \item In case $n \in \{4,8,9 \}$, the assertion comes from
example \ref{exple:4} below.

   \item The exceptional case when $n=6$ is a special case of Theorem \ref{thm:3}.
   \end{enumerate}

\item It is shown [\cite{Byott4}, Corollary 5.7] that any
Hopf-Galois structure on an almost cyclic extension 
of degree $2^r$ with $r\ge 2$ is
of cyclic type. Hence, almost
cyclic extensions of degree $2^r$ with $r\ge 2$ have 
no minimal Hopf-Galois structure.
\end{enumerate}

\subsection{Example 4}\label{exple:4}
On number fields of degree $4,8,$ or $9$.
\begin{enumerate}[1.]
\item Set $N=\mathbb{Z}/2\mathbb{Z} \times  \mathbb{Z}/2\mathbb{Z}$.
The automorphism group of $N$ is $\mathrm{Aut}(N)=\mathrm{GL}_{2}(\mathbf{F}_{\!2})$,
a non-abelian group of order $6$. Let $G'$ be the subgroup of
$\mathrm{Aut}(N)$ generated by  $\left( \begin{array}{ll}
1&1\\
1 & 0 \end{array} \right )$.
It is easily seen that the semidirect product $N\rtimes G'$ is isomorphic to
the alternating group $A_4$.
Denote by $\tilde{K}$ a Galois extension of $\mathbb{Q}$ with Galois group
$N\rtimes G'$, and let $K$ be the fixed field of $G'$.
We know that
$\tilde{K}[N]^{N\rtimes G'}$ defines a minimal Hopf-Galois structure on $K/\mathbb{Q}$.
We also know that any number field of degree $4$
whose normal closure has Galois group equal to $\mathrm{Hol}(N)$ has 
a minimal Hopf-Galois structure. Actually, this is the only one. Indeed,
it is known that the only groups with order $4$, up to isomorphism,
are the Klein four-group and the cyclic group $C_4$.
Since $\mathrm{Hol}(N)$ has order $24$ and $\mathrm{Hol}(C_4)$
has order $8$, our assertion follows from [\cite{CrespoRioVela1}, Theorem 1.5].

\item Set $N=\mathbb{Z}/2\mathbb{Z}
 \times \mathbb{Z}/2\mathbb{Z}  \times \mathbb{Z}/2\mathbb{Z}$ and denote by
 $G'$ the subgroup of
$\mathrm{Aut}(N)=\mathrm{GL}_{3}(\mathbf{F}_{\!2})$
 generated by $ \left( \begin{array}{lll}
1&1&1\\
1 & 1 &0 \\
1&0&0 \end{array} \right )$.
It is easily checked that $G'$ is cyclic
	of order 7, so the semidirect product $G:=N\rtimes G'$ has
	 order 56.
Let $\tilde{K}$ be a Galois extension of $\mathbb{Q}$ with Galois group
$G$, and $K$ the fixed field of $G'$. Then 
$\tilde{K}[N]^{G}$ defines a minimal Hopf-Galois structure on  $K/\mathbb{Q}$.
In addition, we know that any number field of degree $8$
whose normal closure has Galois group equal to $\mathrm{Hol}(N)$ has 
a minimal Hopf-Galois structure. By [\cite{CrespoRioVela1}, Theorem 1.5],
this is the only one
because $\mathrm{Hol}(N)$ is the largest among
all holomorphs of groups with order $8$.

\item  Set $N=\mathbb{Z}/3\mathbb{Z}
 \times  \mathbb{Z}/3\mathbb{Z}$ and denote by $G'$ the subgroup of
$\mathrm{Aut}(N)=\mathrm{GL}_{2}(\mathbf{F}_{\!3})$
 generated by $\left( \begin{array}{ll}
0&1\\
-1 & 0 \end{array} \right )$. 
Note that $G'$ is cyclic of order 4, so the
 semidirect product $G:=N\rtimes G'$ has order 36.
Denote by $\tilde{K}$ a Galois extension of $\mathbb{Q}$ with Galois group
$G$, and let $K$ be the fixed field of $G'$. 
Then
$\tilde{K}[N]^{G}$ defines a minimal Hopf-Galois structure on  $K/\mathbb{Q}$.
By using the same argument as in the first item of the present example,
we see that any number field of degree $9$
whose normal closure has Galois group equal to $\mathrm{Hol}(N)$ has 
only one minimal Hopf-Galois structure.
\end{enumerate}

\subsection{Example 5}\label{exple:5}
Minimal structures from nonabelian characteristically simple groups.\\
We saw that the Galois group of the normal closure of an almost classically
Galois extension is a semidirect product. So 
the study of minimal almost classically Hop-Galois structures
yields the study of normal subgroups of semidirect products.
In \cite{Usenko} Usenko described subgroups of semidirect products.
In particular, he characterized
semidirect products $G = N \rtimes_{\varphi} \! G'$ whose normal subgroups are exhausted by
normal subgroups lying in the centralizer of $N$ in $G$.
We describe here minimal Hopf-Galois structures from normal subgroups
of special semidirect products.
Let $K/k$ be an almost classically Galois extension. Assume that
its Galois closure $\tilde{K}/k$ has Galois group $G=\mathrm{Hol}(N)$ the holomorph of
a nonabelian characteristically simple group $N$.
Then $G$ possesses at least two distinct
normal subgroups which are isomorphic to $N$.
Indeed, it is obviously seen that
$$\Gamma_1:=\{(g; 1)\ \vert \ g\in N\}$$
 is a normal subgroup of $G$ isomorphic to $N$.
On the other hand, denoting by $\mathrm{Inn}(N)$ the group of inner
automorphisms of $N$, we know that 
$$\Gamma_2:=
\{(g^{-1}; \sigma_g )\ \vert \ \sigma_g\in \mathrm{Inn}(N)\}$$
 is a normal subgroup
of $G$ because 
$\theta \circ \sigma_g \circ \theta^{-1}=
\sigma_{\theta(g)}$,
and 
$$\begin{array}{lll}
(x,\theta) \star (g^{-1},\sigma_g) \star (\theta^{-1}(x^{-1}), \theta^{-1}) &
= & \big(x\cdot \theta(g^{-1}),\theta \circ \sigma_g\big) \star \big(\theta^{-1}(x^{-1}), \theta^{-1}\big)\\
 & = & \Bigg(x\cdot \theta\Big(g^{-1}\cdot \sigma_g\big(\theta^{-1}(x^{-1})\big)\Big),
 \theta \circ \sigma_g \circ \theta^{-1}\Bigg)\\
  & = & \big(\theta(g^{-1}), \sigma_{\theta(g)}\big)
\end{array}$$
for all $x,g \in N, \theta \in \mathrm{Aut}(N)$. Note that dot symbol stands for
the group law in $N$, and star symbol stands for the group law in $G$.
It is obvious that $\Gamma_2$ is isomorphic to $N$
and distinct from $\Gamma_1$.

\section{Conclusion and Perspectives}

This work presents Hopf-Galois structures defined 
by cocommutative Hopf algebras $H$ on separable extensions
in the extreme situation when $H$ has only two sub-Hopf algebras.
We first characterized these minimal structures
in Theorem \ref{thm:3}.
Then we specify in Lemma \ref{lemma1} the special case of almost classically Galois
extensions whose normal closures have a Galois group $G$ which is
equal to the holomorph of a characteristically simple group.
We described many illustrations of these two statements. We actually gave examples constructed
from characteristically simple groups, and counterexamples constructed from groups
having a nontrivial proper characteristic subgroup.
The resulting separable field extensions have either no minimal Hopf-Galois structure,
or exactly one minimal structure, or at least two minimal structures.
An interesting problem might be to determine an upper bound of the
number of minimal Hopf-Galois structures on a degree $n$ extension $K/k$ (separable or not)
according to $n$, in the case when the Galois group
of the normal closure $\tilde{K}/k$
is equal to the holomorph of a characteristically simple group $N$.
From classification of characteristically simple groups, one might also
start by computing upper bounds of minimal Hopf-Galois structures of
families of almost classically Galois extensions.
Then we will be able to determine the maximal number of minimal Hopf-Galois structures
which can be defined on a given almost classically Galois extension $K/k$
such that $\mathrm{Gal}(\tilde{K}/k)=\mathrm{Hol}(N)$, according to the size of $N$.

\subsection*{Acknowledgments}
The work reported in this paper is supported by Simons Foundation
via PREMA project, and the International Centre for Theoretical Physics
(ICTP) via their Associate Scheme. The authors would like to
thank the anonymous referee.

\end{document}